\RequirePackage{fix-cm}
\documentclass[smallextended]{svjour3}       
\smartqed  
\usepackage{graphicx}

\usepackage{amsmath,amssymb,amsfonts,amscd}
\usepackage{xcolor}

\numberwithin{theorem}{section} 
\numberwithin{definition}{section}
\numberwithin{remark}{section}
\numberwithin{corollary}{section} 
\numberwithin{lemma}{section} 
\numberwithin{example}{section} 
\numberwithin{equation}{section}

\begin{document}

\title{Generalized Fourier--Feynman transforms and generalized convolution 
products on Wiener space II}


 \titlerunning{Generalized Fourier--Feynman transforms and generalized convolution 
products II}        

\author{Sang Kil Shim \and Jae Gil Choi}
 

 \institute{Sang Kil Shim  \at
           Department of Mathematics, Dankook University,  Cheonan 330-714, Republic of Korea\\
           \email{skshim22@dankook.ac.kr} 
           \and
           Jae Gil   Choi  (Corresponding author)\at
           School of General Education, Dankook University,  Cheonan 330-714, Republic of Korea\\
           \email{jgchoi@dankook.ac.kr}
}

\date{Received: date / Accepted: date}

\maketitle

\begin{abstract}
The purpose of this article is to present the second type  fundamental relationship  
between the generalized Fourier--Feynman transform and  the generalized 
convolution product on Wiener space. The relationships in this article are 
also natural extensions (to the case on an infinite dimensional Banach space) 
of the structure which exists between the Fourier transform and the convolution 
of functions on Euclidean spaces.
\keywords{Wiener space \and
Gaussian process \and
generalized Fourier--Feynman transform \and
generalized  convolution product.}
 \subclass{Primary  46G12;  Secondary   28C20 \and 60G15 \and 60J65}
\end{abstract}

%
\setcounter{equation}{0}
\section{Introduction}\label{sec:introduction}

\par
Given a positive real $T>0$, let $C_0[0,T]$ denote one-parameter Wiener space, 
that is, the space of  all  real-valued continuous functions $x$ on $[0,T]$ with $x(0)=0$.
Let $\mathcal{M}$ denote the class of  all Wiener measurable subsets
of $C_0[0,T]$ and let $\mathfrak{m}$ denote Wiener measure.  
Then, as is  well-known,  $(C_0[0,T],\mathcal{M},\mathfrak{m})$ is 
a complete  measure space.

\par
In  \cite{HPS95,HPS96,HPS97-1,PSS98} Huffman,  Park, Skoug and Storvick  established  fundamental 
relationships between the analytic  Fourier--Feynman transform (FFT) 
and the  convolution product (CP) for  functionals $F$ and $G$ on
$C_0[0,T]$,   as follows: 
\begin{equation}\label{eq:offt-ocp}
T_{q}^{(p)}\big((F*G)_q\big)(y) 
= T_{q}^{(p)}(F)\bigg(\frac{y}{\sqrt2}\bigg)
T_{q}^{(p)}(G)\bigg(\frac{y}{\sqrt2}\bigg)  
\end{equation}
and 
\begin{equation}\label{eq:ocp-offt}
\big(T_{q}^{(p)}(F)*T_{q}^{(p)}(G)\big)_{-q}(y) 
= T_{q}^{(p)}\bigg(F\bigg(\frac{\cdot}{\sqrt2}\bigg)G \bigg(\frac{\cdot}{\sqrt2}\bigg)\bigg)(y)  
\end{equation}
for scale-almost every $y\in C_{0}[0,T]$,
where $T_{q}^{(p)}(F)$ and $(F*G)_q$ denote the $L_p$ analytic FFT and the CP
of functionals $F$ and $G$ on $C_0[0,T]$.
For an elementary introduction of the FFT and the corresponding CP, see \cite{SS04}. 
 
\par
For $f\in L_2(\mathbb R)$, let the Fourier transform of $f$ be given by
\[
\mathcal{F}(f)(u)=\int_{\mathbb R}e^{iuv}f(v)dm_L^{\mathfrak{n}}(v)
\]
and for $f, g\in L_2(\mathbb R)$, let the convolution of $f$ and $g$ be 
given by
\[
(f*g)(u)=\int_{\mathbb R} f(u-v)g(v)dm_L^{\mathfrak{n}}(v)
\]
where $dm_L^{\mathfrak{n}} (v)$ denotes the normalized Lebesgue measure 
$(2\pi)^{-1/2}dv$ on $\mathbb R$. As commented in \cite{Indag}, the Fourier 
transform  $\mathcal{F}$ acts like a homomorphism with convolution $*$ 
and ordinary multiplication on $L_2(\mathbb R)$ as follows:
 for $f, g \in L_2(\mathbb R)$ 
\begin{equation}\label{worthy}
\mathcal{F}(f*g)=\mathcal{F}(f)\mathcal{F}(g).
\end{equation}
But the Fourier transform  $\mathcal{F}$ and the convolution $*$  have a dual 
property such as 
\begin{equation}\label{eq:F-02}
\mathcal{F}(f)*\mathcal{F}(g)=\mathcal{F}(f g).
\end{equation}
Equations \eqref{eq:offt-ocp} and \eqref{eq:ocp-offt} above 
are  natural extensions (to the case on an infinite dimensional Banach space) 
of the equations \eqref{worthy} and \eqref{eq:F-02}, respectively.

\par
In \cite{CCKSY05,HPS97-2}, the authors extended the relationships \eqref{eq:offt-ocp}
and \eqref{eq:ocp-offt} to the cases between the generalized FFT (GFFT) 
and the generalized CP (GCP) of functionals on $C_0[0,T]$.
The definition of the ordinary FFT and the corresponding CP are based on the Wiener integral,
see \cite{HPS95,HPS96,HPS97-1}. While the definition of the GFFT and the GCP studied in \cite{CCKSY05,HPS97-2}
are  based on the generalized Wiener integral \cite{CPS93,PS91}.
The generalized Wiener integral (associated with Gaussian process) 
was defined  by  
$\int_{C_0[0,T]} F(\mathcal Z_h(x,\cdot))d\mathfrak{m}(x)$
where $\mathcal Z_h$ is the Gaussian process on $C_0[0,T]\times[0,T]$ given by
$\mathcal Z_h (x,t)=\int_0^t h(s)\tilde{d}x(s)$,
and where $h$ is a nonzero function in $L_2[0,T]$ and 
$\int_0^t h(s)\tilde{d}x(s)$ 
denotes the Paley--Wiener--Zygmund stochastic integral \cite{PWZ33,Park69,PS88}. 

On the other hand, in \cite{Indag}, the authors defined a more general  CP (see, Definition \ref{def:cp}
below) and  developed the   relationship, such as \eqref{eq:offt-ocp},
between  their GFFT and the GCP (see, Theorem \ref{thm:gfft-gcp-compose} below).
Equation  \eqref{eq:gfft-gcp} in  Theorem \ref{thm:gfft-gcp-compose} is useful
in that it permits one to calculate the GFFT of the GCP of functionals on $C_0[0,T]$
without actually calculating the GCP.

In this paper we work with the second relationship,
such as  equation \eqref{eq:ocp-offt}, between  the GFFT and the GCP of 
functionals  on $C_0[0,T]$.
Our new results  corresponds  to equation  \eqref{eq:F-02} rather 
than equation \eqref{worthy}.
It turns out, as noted in Remark \ref{re:meaning-main} below, that
our second relationship between the GFFT and the CP  also permits one to calculate the GCP of 
the GFFT of functionals on $C_0[0,T]$ without actually calculating the GCP.

\setcounter{equation}{0}
\section{Preliminaries}\label{sec:introduction}
 
\par
In order to present our  relationship between  the GFFT and the GCP, 
we follow the exposition of \cite{Indag}.

\par
A subset $B$ of $C_0[0,T]$ is said to be  scale-invariant measurable
provided $\rho B\in \mathcal{M}$ for all $\rho>0$, and a  scale-invariant 
measurable set $N$ is said to be scale-invariant null provided $\mathfrak{m}(\rho N)=0$ 
for all $\rho>0$. A property that holds except on  a scale-invariant  null set 
is said to hold  scale-invariant almost everywhere (s-a.e.). 
A functional $F$ is said to be scale-invariant measurable provided $F$ is defined on a scale-invariant
measurable set and $F(\rho\,\cdot\,)$ is Wiener-measurable for every $\rho> 0$.
If two functionals 
$F$ and $G$ are equal s-a.e.,  we write $F\approx G$.

\par
Let $\mathbb C$, $\mathbb C_+$ and $\mathbb{\widetilde C}_+$ denote the set of  
complex numbers, complex numbers with positive real part   and  nonzero complex 
numbers with nonnegative real part, respectively. For each $\lambda \in \mathbb C$,
$\lambda^{1/2}$ denotes the principal square root of $\lambda$; i.e., $\lambda^{1/2}$ 
is always chosen to have positive real part, so that  
$\lambda^{-1/2}=(\lambda^{-1})^{1/2}$ is  in $\mathbb C_+$ for 
all $\lambda\in\widetilde{\mathbb C}_+$.

\par
Let $h$ be a function in $L_2[0,T]\setminus\{0\}$ and let $F$ be 
a $\mathbb C$-valued scale-invariant measurable functional on $C_0[0,T]$ 
such that 
\[
\int_{C_0[0,T]} F\big(\lambda^{-1/2}\mathcal Z_h(x,\cdot)\big)d\mathfrak{m}(x) 
=J(h;\lambda)
\]
exists as a finite number for all $\lambda>0$. If there exists a function 
$J^* (h;\lambda)$ analytic on $\mathbb C_+$ such that  
$J^*(h;\lambda)=J(h;\lambda)$ for all $\lambda>0$, then  $J^*(h;\lambda)$ 
is defined to be the generalized analytic  Wiener integral (associated with 
the Gaussian process $\mathcal{Z}_h$) of $F$  over $C_0[0,T]$ with parameter $\lambda$, 
and for $\lambda \in \mathbb C_+$ we write
\[
\int_{C_0[0,T]}^{\mathrm{anw}_{\lambda}} 
F\big(\mathcal{Z}_h(x,\cdot)\big)d\mathfrak{m}(x)
= J^*(h;\lambda).
\]
Let $q\ne 0$ be a real number and let $F$ be a functional such that
\[
\int_{C_0[0,T]}^{\mathrm{anw}_{\lambda}} F\big(\mathcal{Z}_h(x,\cdot)\big)d\mathfrak{m}(x)
\] 
exists for all $\lambda \in \mathbb C_+$. If the following limit exists, we call
it the generalized analytic Feynman integral of $F$ with parameter $q$ and we 
write
\[
\int_{C_0[0,T]}^{\mathrm{anf}_{q}} 
F\big(\mathcal{Z}_h(x,\cdot)\big)d\mathfrak m(x) 
= \lim_{\substack{
\lambda\to -iq \\  \lambda\in \mathbb C_+}}
\int_{C_0[0,T]}^{\mathrm{anw}_{\lambda}} 
F\big(\mathcal{Z}_h(x,\cdot)\big)d\mathfrak m(x).
\]

\par 
Next (see \cite{CCKSY05,Indag,HPS97-2}) we  state the definition of the GFFT.

\renewcommand{\thedefinition}{\thesection.1}
\begin{definition} 
Let $h$ be a function in $L_2[0,T]\setminus\{0\}$.  For $\lambda\in\mathbb{C}_+$ 
and $y \in C_{0}[0,T]$, let
\[
T_{\lambda,h}(F)(y)
=\int_{C_0[0,T]}^{\mathrm{anw}_{\lambda}} 
F\big(y+\mathcal{Z}_h(x,\cdot)\big)d\mathfrak{m}(x). 
\]
For $p\in (1,2]$ we define the $L_p$ analytic GFFT (associated with the Gaussian process  $\mathcal{Z}_h$), 
$T^{(p)}_{q,h}(F)$ of $F$, by the formula,
\[
T^{(p)}_{q,h}(F)(y)
=\operatorname*{l.i.m.}_{\substack{
\lambda\to -iq \\  \lambda\in \mathbb C_+}}
T_{\lambda,h} (F)(y)    				 
\]
if it exists; i.e.,  for each $\rho>0$,
\[
\lim_{\substack{
\lambda\to -iq \\  \lambda\in \mathbb C_+}}
\int_{C_{a,b}[0,T]}\big| T_{\lambda,h} (F)(\rho y)
   -T^{(p)}_{q, h }(F)(\rho y) \big|^{p'} 
d\mathfrak m (y)=0
\]
where $1/p+1/p' =1$. We define the $L_1$ analytic GFFT, $T_{q, h }^{(1)}(F)$ of $F$, 
by the formula  
\[
T_{q, h }^{(1)}(F)(y)
= \lim_{\substack{
\lambda\to -iq \\  \lambda\in \mathbb C_+}}
T_{\lambda,h} (F)(y)
\]
for s-a.e. $y\in C_0[0,T]$ whenever this limit exists.
\end{definition}

\par
We note that for $p \in [1,2]$, $T_{q,h}^{(p)}(F)$ is  defined only s-a.e..
We also note that if $T_{q,h}^{(p)}(F)$  exists  and if $F\approx G$, then 
$T_{q,h}^{(p)}(G)$ exists  and  $T_{q,h}^{(p)}(G)\approx T_{q,h }^{(p)}(F)$.
One can see that for each $h\in L_2[0,T]$, 
$T_{q,h}^{(p)}(F)\approx T_{q,-h}^{(p)}(F)$ since 
\[
\int_{C_0[0,T]}F(x)d\mathfrak{m}(x)=\int_{C_0[0,T]}F(-x)d\mathfrak{m}(x).
\]

\renewcommand{\theremark}{\thesection.2}
\begin{remark}\label{remark:ordinary-fft}
Note that if $h\equiv 1$ on $[0,T]$, then the generalized analytic Feynman 
integral and the  $L_p$ analytic  GFFT, $T_{q,1}^{(p)}(F)$, agree 
with the previous definitions of the analytic Feynman integral and the analytic 
FFT, $T_{q}^{(p)}(F)$, respectively \cite{HPS95,HPS96,HPS97-1,PSS98}  because 
$\mathcal Z_1(x,\cdot)=x$ for all $x \in C_0[0,T]$.
\end{remark}

\par
Next (see \cite{Indag}) we give the definition of our GCP.

\renewcommand{\thedefinition}{\thesection.3}
\begin{definition}\label{def:cp}
Let $F$ and $G$ be  scale-invariant measurable functionals on $C_{0}[0,T]$.
For $\lambda \in \widetilde{\mathbb C}_+$ and $h_1,h_2\in L_2[0,T]\setminus\{0\}$, 
we define their GCP with respect to $\{\mathcal{Z}_{h_1},\mathcal{Z}_{h_2}\}$ 
(if it exists) by
\begin{equation}\label{eq:gcp-Z}
\begin{aligned}
(F*G)_{\lambda}^{(h_1,h_2)}(y)
=
\begin{cases}
\int_{C_0[0,T]}^{\mathrm{ anw}_{\lambda}} 
F\big(\frac{y+{\mathcal Z}_{h_1} (x,\cdot)}{\sqrt2}\big)
G\big(\frac{y-{\mathcal Z}_{h_2} (x,\cdot)}{\sqrt2}\big)d \mathfrak m(x),  
                     \quad    \lambda \in \mathbb C_+ \\
\int_{C_0[0,T]}^{\mathrm{ anf}_{q}} 
F\big(\frac{y+{\mathcal Z}_{h_1} (x,\cdot)}{\sqrt2}\big)
G\big(\frac{y-{\mathcal Z}_{h_2} (x,\cdot)}{\sqrt2}\big)d \mathfrak{m}(x),\\
           \qquad \qquad \qquad  \qquad \qquad 
            \qquad    \lambda=-iq,\,\, q\in \mathbb R, \,\,q\ne 0.
\end{cases} 
\end{aligned}
\end{equation}
When $\lambda =-iq$,  we denote $(F*G)_{\lambda}^{(h_1,h_2)}$ 
by $(F*G)_{q}^{(h_1,h_2)}$.
\end{definition}

\renewcommand{\theremark}{\thesection.4}
\begin{remark}\label{remark:ordinary-cp}
(i) Given  a function  $h$ in $L_2[0,T]\setminus\{0\}$  and letting 
$h_1=h_2\equiv h$, equation \eqref{eq:gcp-Z} yields the convolution
product   studied in \cite{CCKSY05,HPS97-2}: 
\[
\begin{aligned} 
(F*G)_{q}^{(h,h)}(y)&  
\equiv(F*G)_{q,h}(y)\\
&=\int_{C_0[0,T]}^{\mathrm{ anf}_{q}} 
F\bigg(\frac{y+ \mathcal Z_{h} (x,\cdot)}{\sqrt2}\bigg)
G\bigg(\frac{y- \mathcal Z_{h} (x,\cdot)}{\sqrt2}\bigg)d \mathfrak{m}(x) . 
\end{aligned}
\]

(ii) Choosing $h_1=h_2\equiv 1$,  equation \eqref{eq:gcp-Z}  yields 
the convolution product studied in \cite{HPS95,HPS96,HPS97-1,PSS98}:
\[
\begin{aligned}
(F*G)_{q}^{(1,1) }(y)
& \equiv (F*G)_{q}(y)\\
&
=\int_{C_0[0,T]}^{\mathrm{ anf}_{q}} 
F\bigg(\frac{y+ x}{\sqrt2}\bigg)
G\bigg(\frac{y- x}{\sqrt2}\bigg)d \mathfrak{m}(x).
\end{aligned}
\]
\end{remark}

\par
In order to establish our assertion we define the following conventions.
Let $h_1$ and $h_2$ be nonzero functions in $L_2[0,T]$. Then 
there exists a function  $\mathbf{s}\in L_2[0,T]$ such 
that
\begin{equation}\label{eq:fn-rot}
\mathbf{s}^2(t)=h_1^2(t)+h_2^2(t)
\end{equation}
for $m_L$-a.e. $t\in [0,T]$, where $m_L$ denotes  Lebesgue measure on $[0,T]$.
Note that  the function `$\mathbf{s}$' satisfying \eqref{eq:fn-rot} is not 
unique. We will use the symbol  $\mathbf{s}(h_1,h_2)$ for the functions 
`$\mathbf{s}$' that satisfy \eqref{eq:fn-rot} above. 
 Given nonzero functions  $h_1$ and $h_2$  in $L_{2}[0,T]$,
infinitely many functions, $\mathbf{s}(h_1,h_2)$, exist  in $L_{2}[0,T]$. 
Thus $\mathbf{s}(h_1,h_2)$  can be considered as  an equivalence class
of  the  equivalence relation $\sim$ on $L_2[0,T]$ given by 
\[
\mathbf{s}_1\sim \mathbf{s}_2 \,\,\Longleftrightarrow\,\, \mathbf{s}_1^2=\mathbf{s}_2^2
\,\,\,m_L\mbox{-a.e.}.
\]
But we observe  that for every function $\mathbf{s}$ in the 
equivalence class $\mathbf{s}(h_1,h_2)$, the Gaussian random 
variable {\color{red}${\mathcal {Z}}_{\mathbf{s}}(x,T)$} has the normal 
distribution $N(0,\|h_1\|_2^2+\|h_2\|_2^2)$.

Inductively, given a 
sequence   $\mathcal H=\{h_1,\ldots, h_n\}$ of nonzero functions in $L_2[0,T]$, 
let  $\mathbf{s}(\mathcal H)\equiv \mathbf{s}(h_1,h_2,\ldots,h_n)$
be  the equivalence class of  the   functions  $\mathbf{s}$ which satisfy the relation 
\[
\mathbf{s}^2(t)=h_1^2(t)+\cdots+h_n^2(t)
\]
for $m_L$-a.e. $t\in[0,T]$.
Throughout the rest of this paper, for convenience, we will regard
$\mathbf{s}(\mathcal H)$ as a function in $L_2[0,T]$. 
We note that if  the functions $h_1,\ldots, h_n$ are in $L_{\infty}[0,T]$,
then we can take  $\mathbf{s}(\mathcal H)$ to be in $L_{\infty}[0,T]$.
By an induction  argument it  follows that 
\[
\mathbf{s}(\mathbf{s}(h_1,h_2,\ldots,h_{k-1}),h_k)
=\mathbf{s}(h_1,h_2,\ldots,h_k) 
\]
for all $k\in\{2,\ldots,n\}$.

\renewcommand{\theexample}{\thesection.5}
\begin{example}
Let $h_1(t)=t^4$, $h_2(t)=\sqrt{2}t^3$, $h_3(t)=\sqrt{3}t^2$, $h_4(t)={\sqrt{2}}t$, $h_5(t)=1$, 
and $\mathbf{s}(t)=t^4 +t^2 +1$ for $t\in [0,T]$. 
Then $\mathcal H=\{h_1,h_2,h_3,h_4,h_5\}$ is a sequence of functions in $L_2[0,T]$
and it follows that
\[
\mathbf{s}^2(t)= h_1^2(t)+ h_2^2(t)+ h_3^2(t)+ h_4^2(t)+ h_5^2(t).
\]
Thus we can write $\mathbf{s}\equiv \mathbf{s}(h_1,h_2,h_3,h_4,h_5)$.
Furthermore, one can see that
\[
(-1)^{m}\mathbf{s}\equiv \mathbf{s}((-1)^{n_1}h_1,(-1)^{n_2}h_2,(-1)^{n_3}h_3,(-1)^{n_4}h_4,(-1)^{n_5}h_5) 
\]
with $m, n_1,n_2,n_3,n_4,n_5 \in \{1,2\}$.
On the other hand, it  also follows that
\[
\mathbf{s}(h_1,h_2,h_3,h_4,h_5)(t)\equiv \mathbf{s}(g_1,g_2,g_3) (t)
\]
for each $t\in [0,T]$,
where $g_1(t)=-t^4-1$, $g_2(t)={\sqrt2} t\sqrt{t^4+1}$, and $g_3(t)=t^2$ for $t\in [0,T]$.
 \end{example}

\renewcommand{\theexample}{\thesection.6}

\begin{example}
Let $h_1(t)=t^4+t^2$,  $h_2(t)=t^4-t^2$,  $h_3(t)=\sqrt2 t^3$, 
and  $\mathbf{s}(t)=\sqrt{2(t^8 +t^4)}$ for $t\in [0,T]$. 
Then, by the convention for $\mathbf{s}$, it follows that
\[
\mathbf{s}(t)\equiv\mathbf{s}(h_1,h_2) (t)
  \equiv \mathbf{s}({\sqrt2} h_2 ,{\sqrt2} h_3)(t).
\]
\end{example}

\renewcommand{\theexample}{\thesection.7}
\begin{example}
Using the well-known formulas for trigonometric and hyperbolic functions,
it follows that 
\[
\begin{aligned}
\sec \big(\tfrac{\pi}{4 T} t\big) 
&=\mathbf{s}\big(1, \tan\big(\tfrac{\pi}{4 T}  \cdot\big)\big)(t)\\
&=\mathbf{s}\big(\sin,\cos,\tan\big(\tfrac{\pi}{4 T}  \cdot\big)\big)(t) \\
&
=\mathbf{s}\big(\sin\big(\tfrac{\pi}{4 T}  \cdot\big),\cos\big(\tfrac{\pi}{4 T} 
 \cdot\big),\tan\big(\tfrac{\pi}{4 T}  \cdot\big)\big)(t),
\end{aligned}
\]
\[
\cosh t =\mathbf{s}(1, \sinh)(t)=\mathbf{s}(-1, \sinh)(t)=\mathbf{s}(\sin,\cos,\sinh)(t) ,
\]
and
\[
-\coth \big(t+\tfrac12\big) =\mathbf{s}\big(1, \mathrm{csch}\big(\cdot+\tfrac12\big)\big)(t)
=\mathbf{s}(-\sin,\cos,- \mathrm{csch}\big(\cdot+\tfrac12\big))(t) 
\]
for each $t\in [0,T]$.
\end{example}

%
\setcounter{equation}{0}
\section{The  relationship  between the GFFT and the GCP}

\par
The Banach algebra $\mathcal S(L_2[0,T])$ consists of functionals 
on $C_0[0,T]$ expressible in the form 
\begin{equation}\label{eq:element}
F(x)=\int_{L_2[0,T]}\exp\{i\langle{u,x}\rangle\}df(u)
\end{equation}
for s-a.e. $x\in C_0[0,T]$, where the associated measure $f$ is an 
element of $\mathcal M(L_2[0,T])$, the space of  $\mathbb C$-valued 
countably additive (and hence finite) Borel measures on $L_2[0,T]$, 
and the  pair $\langle{u,x}\rangle$ denotes the Paley--Wiener--Zygmund 
stochastic integral $\mathcal Z_u(x,T) \equiv \int_0^T u(s)\tilde{d}x(t)$. 
For more details, see \cite{CS80,CPS93,HPS97-2,PSS98}.
 
\par
We first present two  known results for the GFFT and the GCP 
of functionals in the Banach algebra $\mathcal S(L_2[0,T])$. 

\renewcommand{\thetheorem}{\thesection.1}
\begin{theorem}[\cite{HPS97-2}]\label{thm:gfft}
Let $h$ be a nonzero function in $L_\infty[0,T]$, and let 
$F\in\mathcal S(L_2[0,T])$ be given by equation \eqref{eq:element}. Then,  
for all $p\in[1,2]$, the $L_p$ analytic  GFFT, $T_{q,h}^{(p)}(F)$ of $F$ 
exists for all nonzero real numbers $q$, belongs to $\mathcal S(L_2[0,T])$, 
and is given by the formula
\[
T_{q,h}^{(p)}(F)(y)
= \int_{L_2[0,T]}\exp\{i\langle{u,y}\rangle\}df_t^h(u)
\]
for s-a.e. $y\in C_{0}[0,T]$, where  $f_t^h$ is the complex measure  in 
$\mathcal M(L_2[0,T])$ given by
\[
f_t^{h}(B)=\int_B \exp\bigg\{-\frac{i}{2q}\|uh\|_2^2\bigg\}df(u)
\]
for $B \in \mathcal B(L_2[0,T])$.
\end{theorem}

\renewcommand{\thetheorem}{\thesection.2}
\begin{theorem}[\cite{Indag}] \label{thm:gcp}
Let $k_1$ and $k_2$ be nonzero functions in $L_\infty[0,T]$ and let $F$ 
and $G$ be elements of $\mathcal S(L_2[0,T])$ with corresponding finite 
Borel measures $f$ and $g$ in $\mathcal M(L_2[0,T])$. Then, the GCP 
$(F*G)_q^{(k_1,k_2)}$ exists for all nonzero real $q$, belongs to 
$\mathcal S(L_2[0,T])$, and is given by the formula
\[
(F*G)_q^{(k_1,k_2)}(y)
= \int_{L_2[0,T]}\exp\{i\langle{w,y}\rangle\}d\varphi^{k_1,k_2}_c(w)
\]
for s-a.e. $y\in C_{0}[0,T]$, where  
\[
\varphi^{k_1,k_2}_c
=\varphi^{k_1,k_2}\circ\phi^{-1},
\]
$\varphi^{k_1,k_2}$ is the complex measure in $\mathcal M(L_2[0,T])$ given 
by
\[
\varphi_{k_1,k_2}(B)
=\int_B \exp\bigg\{-\frac{i}{4q}\|uk_1-vk_2\|_2^2\bigg\}df(u)dg(v)
\]
for $B \in \mathcal B(L_2^2[0,T])$, and $\phi:L_2^2[0,T]\to L_2[0,T]$ is the 
continuous function  given by $\phi(u,v)=(u+v)/\sqrt2$.
\end{theorem}

\par
The following corollary and theorem will be very useful to prove our main theorem
(namely, Theorem \ref{thm:cp-tpq02}) which  we establish 
the relationship between the GFFT and the GCP such as equation \eqref{eq:ocp-offt}. 
The following corollary is a simple consequence of Theorem \ref {thm:gfft}.

\renewcommand{\thecorollary}{\thesection.3}

\begin{corollary}\label{thm:afft-inverse}
Let $h$ and  $F$ be  as in Theorem \ref{thm:gfft}.
Then, for all $p\in[1,2]$,   and all nonzero real $q$, 
\begin{equation}\label{eq:inverse}
T_{-q, h}^{(p)}\big(T_{q,h}^{(p)}(F)\big)\approx F.
\end{equation}
As such, the  GFFT, $T_{q,h}^{(p)}$, has the 
inverse transform $\{T_{q,h}^{(p)}\}^{-1}=T_{-q,h}^{(p)}$.
\end{corollary}

\par
The following theorem is due to Chang, Chung and Choi \cite{Indag}.

\renewcommand{\thetheorem}{\thesection.4}

\begin{theorem} \label{thm:gfft-gcp-compose}
Let $k_1$, $k_2$, $F$,  and $G$ be as in Theorem \ref{thm:gcp},
and let $h$ be a  nonzero function in $L_\infty[0,T]$. Assume that $h^2=k_1k_2$ $m_L$-a.e. 
on $[0,T]$. Then,  for all $p\in[1,2]$ and all nonzero real $q$, 
\begin{equation}\label{eq:gfft-gcp}
\begin{aligned}
&T_{q,h}^{(p)}\big((F*G)_q^{(k_1,k_2)}\big)(y) \\
& = T_{q,\mathbf{s}(h,k_1)/\sqrt2}^{(p)}(F)\bigg(\frac{y}{\sqrt2}\bigg)
T_{q,\mathbf{s}(h,k_2)/\sqrt2}^{(p)}(G)\bigg(\frac{y}{\sqrt2}\bigg)  
\end{aligned}
\end{equation}
for s-a.e. $y\in C_{0}[0,T]$, where 
$\mathbf{s}(h,k_j)$'s, $j\in \{1,2\}$, are the functions which satisfy the relation
\eqref{eq:fn-rot}, respectively. 
\end{theorem}

\renewcommand{\theremark}{\thesection.5}

\begin{remark}
In equation \eqref{eq:gfft-gcp}, choosing $h=k_1=k_2\equiv 1$
yields equation \eqref{eq:offt-ocp} above.
Also, letting  $h=k_1=k_2$ yields the results studied in \cite{CCKSY05,HPS97-2}.
As mentioned above, equation \eqref{eq:gfft-gcp} is  a more  general extension of equation \eqref{worthy}
to the  case on an infinite dimensional Banach space.  
\end{remark}

\par
We are now ready to establish our  main theorem in this paper.

\renewcommand{\thetheorem}{\thesection.6}

\begin{theorem}    \label{thm:cp-tpq02}
Let $k_1$, $k_2$, $F$,  $G$, and $h$ be as in Theorem  \ref{thm:gfft-gcp-compose}.
Then,  for all $p\in[1,2]$ and all nonzero real $q$, 
\begin{equation}\label{eq:cp-fft-basic}
\begin{aligned}
&\Big(T_{q,\mathbf{s}(h,k_1)/\sqrt{2} }^{(p)} (F)*
 T_{q,\mathbf{s}(h,k_2)/\sqrt{2}}^{(p)}(G) \Big)_{-q}^{(k_1,k_2)}(y) \\
&=T_{q,h}^{(p)} \bigg(F \bigg(\frac{\cdot}{\sqrt2} \bigg)
 G \bigg(\frac{\cdot}{\sqrt2}\bigg)\bigg)(y)
\end{aligned}
\end{equation}
for s-a.e. $y\in C_0[0,T]$, where 
$\mathbf{s}(h,k_j)$'s, $j\in \{1,2\}$, are the functions which satisfy the relation
\eqref{eq:fn-rot}, respectively.
\end{theorem}
\begin{proof}
Applying \eqref{eq:inverse}, \eqref{eq:gfft-gcp} with $F$, $G$, and $q$
replaced with $T_{q,\mathbf{s}(h,k_1)/\sqrt{2} }^{(p)} (F)$, 
$T_{q,\mathbf{s}(h,k_2)/\sqrt{2}}^{(p)}(G)$,
and $-q$, respectively, and \eqref{eq:inverse} again, 
it follows that for s-a.e. $y\in C_0[0,T]$,
\[
\begin{aligned}
&\Big(T_{q,\mathbf{s}(h,k_1)/\sqrt{2} }^{(p)} (F)*
 T_{q,\mathbf{s}(h,k_2)/\sqrt{2}}^{(p)}(G) \Big)_{-q}^{(k_1,k_2)}(y)\\
&= T_{q,h}^{(p)}\Big(T_{-q,h}^{(p)}\Big(\big(T_{q,\mathbf{s}(h,k_1)/\sqrt{2} }^{(p)} (F)*
 T_{q,\mathbf{s}(h,k_2)/\sqrt{2}}^{(p)}(G) \big)_{-q}^{(k_1,k_2)}\Big)\Big)(y)\\
&= T_{q,h}^{(p)}\bigg(
T_{-q,\mathbf{s}(h,k_1)/\sqrt2}^{(p)}\Big(T_{q,\mathbf{s}(h,k_1)/\sqrt{2} }^{(p)}(F)\Big)
\bigg(\frac{\cdot}{\sqrt2}\bigg)\\
&\qquad\quad\times
T_{-q,\mathbf{s}(h,k_2)/\sqrt2}^{(p)}\Big(T_{q,\mathbf{s}(h,k_2)/\sqrt{2} }^{(p)}(G)\Big)
\bigg(\frac{\cdot}{\sqrt2}\bigg)\bigg)
(y)\\
&=T_{q,h}^{(p)} \bigg(F \bigg(\frac{\cdot}{\sqrt2} \bigg)
 G \bigg(\frac{\cdot}{\sqrt2}\bigg)\bigg)(y)
\end{aligned}
\]
as desired.
\qed\end{proof}

\renewcommand{\theremark}{\thesection.7}

\begin{remark}\label{re:meaning-main}
(i) Equation \eqref{eq:gfft-gcp}  shows that the GFFT of the GCP of two functionals 
is the ordinary product of their transforms.
On the other hand, equation \eqref{eq:cp-fft-basic}  above shows that 
the GCP of GFFTs of functionals is the GFFT of product of the functionals. 
These equations are useful in that they permit one to calculate 
$T_{q,h}^{(p)}((F*G)_q^{(k_1,k_2)})$ and 
$(T_{q,\mathbf{s}(h,k_1)/\sqrt{2} }^{(p)} (F)* T_{q,\mathbf{s}(h,k_2)/\sqrt{2}}^{(p)}(G))_{-q}^{(k_1,k_2)}$
without actually calculating the GCPs involved them, respectively. 
In practice, equation \eqref{eq:cp-fft-basic} tells us that
to calculate $T_{q,h}^{(p)}(F(\frac{\cdot}{\sqrt2}) G(\frac{\cdot}{\sqrt2} ))$
is easier to calculate
than are   $T_{q,\mathbf{s}(h,k_1)/\sqrt{2} }^{(p)} (F)$,
$T_{q,\mathbf{s}(h,k_1)/\sqrt{2} }^{(p)} (G)$, and
$(T_{q,\mathbf{s}(h,k_1)/\sqrt{2} }^{(p)} (F)* T_{q,\mathbf{s}(h,k_2)/\sqrt{2}}^{(p)}(G) )_{-q}^{(k_1,k_2)}$.

(ii) Equation \eqref{eq:cp-fft-basic}  is  a more  general extension of equation \eqref{eq:F-02}
to the  case on an infinite dimensional Banach space.  
\end{remark}

\renewcommand{\thecorollary}{\thesection.8}

\begin{corollary}[Theorem  3.1  in \cite{PSS98}] 
Let $F$ and $G$ be as in Theorem  \ref{thm:gcp}. Then, for all $p\in[1,2]$ 
and all real $q\in\mathbb R\setminus\{0\}$,
\[
\Big( T_q^{(p)}(F)*T_q^{(p)}(G)\Big)_{-q} (y)
=T_q^{(p)}\bigg(F \bigg(\frac{\cdot}{\sqrt2}\bigg)
 G \bigg(\frac{\cdot}{\sqrt2}\bigg) \bigg)(y)
\]
for s-a.e. $y\in C_{0}[0,T]$, where $T_q^{(p)}(F)$ denotes the ordinary 
analytic FFT of $F$  and $(F*G)_q$ 
denotes the CP of $F$ and $G$  (see   Remarks \ref{remark:ordinary-fft} 
and \ref{remark:ordinary-cp}).
\end{corollary}
\begin{proof}
In equation \eqref{eq:cp-fft-basic}, simply choose  $h=k_1=k_2\equiv 1$.
\qed\end{proof}

\renewcommand{\thecorollary}{\thesection.9}

\begin{corollary}[Theorem  3.2  in \cite{CCKSY05}] 
Let  $F$, $G$, and $h$  be as in Theorem  \ref{thm:gfft-gcp-compose}. 
Then, for all $p\in[1,2]$ and all real $q\in\mathbb R\setminus\{0\}$,
\[
\Big(T_{q,h}^{(p)}(F)*T_{q,h}^{(p)}(G)\Big)_{-q} (y)
=T_{q,h}^{(p)}\bigg(F \bigg(\frac{\cdot}{\sqrt2}\bigg)
 G \bigg(\frac{\cdot}{\sqrt2}\bigg) \bigg)(y)
\]
for s-a.e. $y\in C_{0}[0,T]$, where $(F*G)_q\equiv (F*G)_q^{(h,h)}$ denotes the GCP of $F$ 
and $G$ studied in    \cite{CCKSY05,HPS97-2}  (see   Remark \ref{remark:ordinary-cp}).
\end{corollary}
\begin{proof}
In equation \eqref{eq:cp-fft-basic}, simply choose  $h=k_1=k_2$.
\qed\end{proof}

\setcounter{equation}{0}
\section{Examples}

The assertion  in Theorem \ref{thm:cp-tpq02} above can be applied to many 
Gaussian processes $\mathcal Z_h$  with $h\in L_\infty[0,T]$. 
In view of the assumption  in Theorems \ref{thm:gfft-gcp-compose} and \ref{thm:cp-tpq02}, 
we have to check that 
there exist solutions $\{h,k_1,k_2, \mathbf{s}_1,\mathbf{s}_2\}$  of the system
\[
\begin{cases}
\mbox{(i)}   &h^2 =k_1k_2,\\
\mbox{(ii)}  &\mathbf{s}_1=\mathbf{s}(h,k_1) \,\, m_L\mbox{-a.e} \, \mbox{  on }\, [0,T],\\
\mbox{(iii)} &\mathbf{s}_2=\mathbf{s}(h,k_2) \,\, m_L\mbox{-a.e} \, \mbox{  on }\, [0,T], 
\end{cases}
\]
or, equivalently,
\begin{equation}\label{system}
\begin{cases}
\mbox{(i)}   &h^2 =k_1k_2,\\
\mbox{(ii)}  &\mathbf{s}_1^2=h^2 +k_1^2 \,\, m_L\mbox{-a.e} \, \mbox{  on }\, [0,T],\\
\mbox{(iii)} &\mathbf{s}_2^2=h^2 +k_2^2 \,\, m_L\mbox{-a.e} \, \mbox{  on }\, [0,T].
\end{cases}
\end{equation}
Throughout this section, we will present some examples for 
the solution sets of the system \eqref{system}.
To do this we consider the Wiener space $C_0[0,1]$
and the Hilbert space $L_2[0,1]$ for simplicity.

\renewcommand{\theexample}{\thesection.1}

\begin{example} (Polynomials)
The set $\mathcal P =\{h, k_1, k_2, \mathbf{s}_1,\mathbf{s}_2\}$ 
of functions in $L_\infty[0,1]$ with 
\[
\begin{cases}
& h(t) = 2t(t^2-1)  \\
&k_1(t) =(t^2-1)^2,  \\
&k_2(t) =4t^2,  \\ 
&\mathbf{s}_1(t) = (t^2-1)(t^2+1),  \\ 
&\mathbf{s}_2(t) = 2 t(t^2+1)
\end{cases}
\] is a solution  set  of the system \eqref{system}.
Thus
\[
\mathbf{s} (h,k_1)(t) \equiv \mathbf{s}_1(t) 
=(t^2-1)(t^2+1),
\]
and
\[
\mathbf{s}(h,k_2)(t)\equiv \mathbf{s}_2(t) 
=2 t(t^2+1)
\] 
for all $t\in [0,1]$. In this case,  equation \eqref{eq:cp-fft-basic} 
with the functions in $\mathcal P$ holds for any functionals in $F$ and $G$ 
in $\mathcal S(L_2[0,1])$.
\end{example}

\renewcommand{\theexample}{\thesection.2}

\begin{example}  (Trigonometric  functions I) 
The set $\mathcal T_1=\{h, k_1, k_2, \mathbf{s}_1,\mathbf{s}_2\}$ 
of functions in $L_\infty[0,1]$ with 
\[
\begin{cases}
 h(t)=\sin 2t=2\sin t\cos t,  \\
k_1(t)=2\sin^2t,  \\
k_2(t)=2\cos^2t,  \\ 
\mathbf{s}_1(t)=2\sin t,  \\ 
\mathbf{s}_2(t)=2\cos t
\end{cases}
\] is a solution set  of the system \eqref{system}.
Thus
\[
\mathbf{s} (h,k_1)(t) \equiv \mathbf{s}_1(t) 
=\mathbf{s}(2\sin\cos,2\sin^2)(t)
=2 \sin t, 
\]
and
\[
\mathbf{s}(h,k_2)(t)\equiv \mathbf{s}_2(t) 
=\mathbf{s}(2\sin\cos,2\cos^2)(t)
=2 \cos t 
\] 
for all $t\in [0,1]$. Also, using equation \eqref{eq:cp-fft-basic},
it follows that for all $p\in[1,2]$, all nonzero real $q$, and all functionals $F$ and $G$ in 
 $\mathcal S(L_2[0,1])$,
\[
\Big(T_{q,\sqrt{2}\sin }^{(p)} (F)*
 T_{q,\sqrt{2}\cos}^{(p)}(G) \Big)_{-q}^{(2\sin^2,2\cos^2)}(y) 
=T_{q,2\sin\cos}^{(p)} \bigg(F \bigg(\frac{\cdot}{\sqrt2} \bigg)
 G \bigg(\frac{\cdot}{\sqrt2}\bigg)\bigg)(y)
\]
for s-a.e. $y\in C_0[0,1]$.
\end{example}

\renewcommand{\theexample}{\thesection.3}
\begin{example}  (Trigonometric  functions II)
 The set $\mathcal T_2=\{h, k_1, k_2, \mathbf{s}_1,\mathbf{s}_2\}$ 
of functions in $L_\infty[0,1]$ with 
\[
\begin{cases}
 h(t)=\sqrt2\sin t,  \\
k_1(t)=\sqrt2\sin t\tan t,  \\
k_2(t)=\sqrt2\cos t,  \\ 
\mathbf{s}_1(t)=\sqrt2\tan t,  \\ 
\mathbf{s}_2(t)=\sqrt2
\end{cases}
\] 
is a solution  set  of the system \eqref{system}.
Thus
\[
\mathbf{s} (h,k_1)(t) \equiv \mathbf{s}_1(t) 
=\mathbf{s}(\sqrt2\sin,\sqrt2\sin\tan)(t)
=  \sqrt2\tan t, 
\]
and
\[
\mathbf{s}(h,k_2)(t)\equiv \mathbf{s}_2(t) 
=\mathbf{s}(\sqrt2\sin,\sqrt2\cos)(t)
=\sqrt2 \,\,\,\,\,(\mbox{constant function})
\] 
for all $t\in [0,1]$. 
\end{example}

\renewcommand{\theexample}{\thesection.4}

\begin{example} (Hyperbolic functions)
The hyperbolic functions are defined in terms of the exponential functions
$e^{x}$  and $e^{-x}$. The set $\mathcal H=\{h, k_1, k_2, \mathbf{s}_1,\mathbf{s}_2\}$ 
of functions in $L_\infty[0,1]$ with 
\[
\begin{cases}
 h(t)=1 ,  \\
k_1(t)= \sinh\big(t+\tfrac12\big),  \\
k_2(t)= \mathrm{csch} \big(t+\tfrac12\big),  \\ 
\mathbf{s}_1(t)= \cosh\big(t+\tfrac12\big),  \\ 
\mathbf{s}_2(t)= \coth\big(t+\tfrac12\big)
\end{cases}
\] is a solution  set of the system \eqref{system}.
Thus
\[
\mathbf{s} (h,k_1)(t) \equiv \mathbf{s}_1(t) 
=\mathbf{s}\big(1,\sinh\big(\cdot+\tfrac12\big)\big)(t)
=  \cosh\big(t+\tfrac12\big), 
\]
and
\[
\mathbf{s}(h,k_2)(t)\equiv \mathbf{s}_2(t) 
=\mathbf{s}\big(1,\mathrm{csch} \big(\cdot+\tfrac12\big)\big)(t)
= \coth\big(t+\tfrac12\big)
\] 
for all $t\in [0,1]$. 
\end{example}

\setcounter{equation}{0} 
\section{Iterated GFFTs  and GCPs}\label{relation2}

\par
In this section, we present  general relationships 
between the iterated GFFT and the GCP for functionals 
in $\mathcal S(L_2[0,T])$ which  are developments of    \eqref{eq:cp-fft-basic}.
To do this we quote a  result  from \cite{Indag}.

\renewcommand{\thetheorem}{\thesection.1}

\begin{theorem}\label{thm:2018-step1}
Let $F\in \mathcal S(L_2[0,T])$  be given by equation \eqref{eq:element}, and 
let $\mathcal H=\{h_1,\ldots,h_n\}$ be  a finite sequence  of nonzero  functions in $L_\infty[0,T]$.
Then, for all $p\in[1,2]$ and all nonzero real $q$, the iterated $L_p$ analytic  GFFT, 
\[
T_{q,h_n}^{(p)}\big(T_{q,h_{n-1}}^{(p)}\big(
   \cdots\big(T_{q,h_2}^{(p)}\big(T_{q,h_1}^{(p)}(F)\big)\big)\cdots\big)\big)
\]
of $F$ exists, belongs to $\mathcal S(L_2[0,T])$, and is given by the formula
\[
T_{q,h_n}^{(p)}\big(T_{q,h_{n-1}}^{(p)}\big(
   \cdots\big(T_{q,h_2}^{(p)}\big(T_{q,h_1}^{(p)}(F)\big)\big)\cdots\big)\big) (y)
= \int_{L_2[0,T]}\exp\{i\langle{u,y}\rangle\}df_t^{h_1,\ldots,h_n}(u)
\]
for s-a.e. $y\in C_{0}[0,T]$, where  $f_t^{h_1,\ldots,h_n}$ is the complex 
measure  in  $\mathcal M(L_2[0,T])$ given by
\[
f_t^{h_1,\ldots,h_n}(B)
=\int_B \exp\bigg\{-\frac{i}{2q}\sum_{j=1}^n\|uh_j\|_2^2\bigg\}df(u)
\]
for $B \in \mathcal B(L_2[0,T])$. Moreover it follows that
\begin{equation}\label{eq:gfft-n-fubini-add}
T_{q,h_n}^{(p)}\big(T_{q,h_{n-1}}^{(p)}\big(
   \cdots\big(T_{q,h_2}^{(p)}\big(T_{q,h_1}^{(p)}(F)\big)\big)\cdots\big)\big)(y)  
=T_{q, \mathbf{s}(\mathcal H)}^{(p)}(F)(y)
\end{equation}
for s-a.e. $y\in C_{0}[0,T]$, where 
$\mathbf{s}(\mathcal H)\equiv \mathbf{s}(h_1,\ldots,h_n)$ is a function in $L_{\infty}[0,T]$
satisfying the relation
\begin{equation}\label{eq:fn-rot-ind}
\mathbf{s}(\mathcal H)^2(t)=h_1^2(t)+\cdots+h_n^2(t)
\end{equation} 
 for $m_L$-a.e. $t\in [0,T]$.
\end{theorem}

 We next establish two types of extensions   of Theorem \ref{thm:cp-tpq02} above. 

\renewcommand{\thetheorem}{\thesection.2}

\begin{theorem} \label{thm:iter-gfft-gcp-compose}
Let  $k_1$, $k_2$, $F$, and $G$ be as in 
Theorem \ref{thm:gcp}, and let $\mathcal H=\{h_1,\ldots,h_n\}$ 
be  a finite sequence of nonzero  functions in $L_{\infty}[0,T]$. 
Assume that  
\[
\mathbf{s}(\mathcal H)^2  \equiv \mathbf{s} (h_1,\ldots,h_n)^2=k_1k_2
\] 
for $m_L$-a.e. on $[0,T]$, where $\mathbf{s}(\mathcal H)$ is the function in $L_{\infty}[0,T]$
satisfying \eqref{eq:fn-rot-ind} above. 
Then,  for all $p\in[1,2]$ and all nonzero real $q$, 
\begin{equation} \label{eq:multi-rel-01}
\begin{aligned}
&\Big(T_{q,k_1/\sqrt2}^{(p)}
\big(T_{q,h_n/\sqrt2}^{(p)}\big(\cdots\big(T_{q,h_2/\sqrt2}^{(p)}
\big(T_{q,h_1/\sqrt2}^{(p)}(F)\big)\big)\cdots\big)\big)\\
&\qquad 
*T_{q,k_2/\sqrt2}^{(p)}\big( T_{q,h_n/\sqrt2}^{(p)}\big( \cdots\big(T_{q,h_2/\sqrt2}^{(p)}
\big(T_{q,h_1/\sqrt2}^{(p)}(G)\big)\big)\cdots\big)\big)\Big)_{-q}^{(k_1,k_2)}(y)\\
&=\Big(T_{q, \mathbf{s}(\mathcal H,k_1)/\sqrt2}^{(p)}(F)
*T_{q, \mathbf{s}(\mathcal H,k_2)/\sqrt2}^{(p)}(G)\Big)_{-q}^{(k_1,k_2)}(y)\\
&= T_{q,\mathbf{s}(\mathcal H)}^{(p)} \bigg(F \bigg(\frac{\cdot}{\sqrt2} \bigg)
 G \bigg(\frac{\cdot}{\sqrt2}\bigg)\bigg)(y)
\end{aligned}
\end{equation}
for s-a.e. $y\in C_0[0,T])$, where   $\mathbf{s}(\mathcal H,k_1)$ 
and $\mathbf{s}(\mathcal H,k_2)$
are functions in $L_{\infty}[0,T]$ satisfying the relations
\[
\mathbf{s}(\mathcal H,k_1)^2
\equiv \mathbf{s}(h_1,\ldots,h_n,k_1)^2
=h_1^2 +\cdots+h_n^2 +k_1^2
\]
and
\[
\mathbf{s}(\mathcal H,k_2)^2\equiv \mathbf{s}(h_1,\ldots,h_n,k_2)^2
=h_1^2 +\cdots+h_n^2 +k_2^2 
\]
for $m_L$-a.e. on $[0,T]$, respectively.
\end{theorem}
 \begin{proof}
Applying \eqref{eq:gfft-n-fubini-add}, the first equality of \eqref{eq:multi-rel-01}
follows immediately. Next using \eqref{eq:cp-fft-basic} with $h$ replaced with $\mathbf{s}(\mathcal H)$,
the second equality of \eqref{eq:multi-rel-01} also follows.
\qed\end{proof}

   \par
 In view of equations  \eqref{eq:gfft-n-fubini-add} and \eqref{eq:cp-fft-basic},
 we also obtain the following assertion.

\renewcommand{\thetheorem}{\thesection.3}

\begin{theorem} \label{thm:iter-gfft-gcp-compose-2nd}
Let $F$ and $G$ be as in Theorem  \ref{thm:gcp}. 
Given a nonzero function  $h$ in $L_{\infty}[0,T]$
and finite sequences $\mathcal K_1=\{k_{11},k_{12},\ldots,k_{1n}\}$
and  $\mathcal K_2=\{k_{21},k_{22},\ldots,k_{2m}\}$ of   nonzero  functions in $L_{\infty}[0,T]$, 
assume that  
\[
h^2=\mathbf{s}(\mathcal K_1)\mathbf{s}(\mathcal K_2)
\] 
for $m_L$-a.e. on $[0,T]$. 
Then,  for all $p\in[1,2]$ and all nonzero real $q$, 
\begin{equation} \label{eq:multi-rel-02-2nd}
\begin{aligned}
&\Big(T_{q,h/\sqrt2}^{(p)}
\big(T_{q,k_{1n}/\sqrt2}^{(p)}  
\big(\cdots
\big(T_{q,k_{12}/\sqrt2}^{(p)}\big(T_{q,k_{11}/\sqrt2}^{(p)}(F)\big)\big)\cdots\big)\big)\big)\\
&\quad
*T_{q,h/\sqrt2}^{(p)}\big( T_{q,k_{2m}/\sqrt2}^{(p)}  
\big( \cdots
\big(T_{q,k_{22}/\sqrt2}^{(p)}\big(T_{q,k_{21}/\sqrt2}^{(p)}(G)\big)\big)\cdots\big)\big)\big)\Big)_{-q}^{(\mathbf{s}(\mathcal K_1),\mathbf{s}(\mathcal K_2))}(y)\\
&=\Big(T_{q,h/\sqrt2}^{(p)}\big(T_{q,\mathbf{s}(\mathcal K_1)/\sqrt2}^{(p)}(F)\big) 
*T_{q,h/\sqrt2}^{(p)}\big( T_{q,\mathbf{s}(\mathcal K_2)/\sqrt2}^{(p)}(G)\big) \Big)_{-q}^{(\mathbf{s}(\mathcal K_1),\mathbf{s}(\mathcal K_2))}(y)\\
&=\Big(T_{q, \mathbf{s}(h,\mathbf{s}(\mathcal K_1))/\sqrt2}^{(p)}(F)
*T_{q,\mathbf{s}(h,\mathbf{s}(\mathcal K_2))/\sqrt2}^{(p)}(G)\Big)_{-q}^{(\mathbf{s}(\mathcal K_1),\mathbf{s}(\mathcal K_2))}(y)\\
&= T_{q,h}^{(p)} \bigg(F \bigg(\frac{\cdot}{\sqrt2} \bigg)
 G \bigg(\frac{\cdot}{\sqrt2}\bigg)\bigg)(y)
\end{aligned}
\end{equation}
for s-a.e. $y\in C_0[0,T])$, where 
$\mathbf{s}(h,\mathbf{s}(\mathcal K_1))$, and $\mathbf{s}(h,\mathbf{s}(\mathcal K_2))$
are functions in $L_{\infty}[0,T]$ satisfying the relations
\[
\mathbf{s}(h,\mathbf{s}(\mathcal K_1))^2
=h^2 +\mathbf{s} (\mathcal K_1)^2=h^2+k_{11}^2 + \cdots+k_{1n}^2, 
\]
and
\[
\mathbf{s}(h,\mathbf{s}(\mathcal K_2))^2
=h^2 +\mathbf{s} (\mathcal K_2)^2=h^2+k_{21}^2   +\cdots+k_{2m}^2
\]
for $m_L$-a.e. on $[0,T]$, respectively.
\end{theorem}

\renewcommand{\theremark}{\thesection.4}

\begin{remark}
Note that given the functions $\{\mathbf{s}(\mathcal H),k_1,k_2, \mathbf{s}(\mathcal H,k_1),\mathbf{s}(\mathcal H,k_2) \}$
in Theorem \ref{thm:iter-gfft-gcp-compose}, the set $\mathcal F=\{h, k_1, k_2, \mathbf{s}_1,\mathbf{s}_2\}$ 
of functions in $L_\infty[0,T]$ with 
\[
\begin{cases}
 h(t)=\mathbf{s}(\mathcal H)(t),  \\
\mathbf{s}_1(t)=\mathbf{s}(\mathcal H,k_1)(t),  \\ 
\mathbf{s}_2(t)=\mathbf{s}(\mathcal H,k_2)(t)
\end{cases}
\] 
is a solution set of the system \eqref{system}.
Also,  given the functions 
\[
\{h,\mathbf{s}(\mathcal K_1),\mathbf{s}(\mathcal K_2),\mathbf{s}(h,\mathbf{s}(\mathcal K_1)),\mathbf{s}(h,\mathbf{s}(\mathcal K_2))\}
\]
in Theorem \ref{thm:iter-gfft-gcp-compose-2nd}, the set $\mathcal F=\{h, k_1, k_2, \mathbf{s}_1,\mathbf{s}_2\}$ 
of functions in $L_\infty[0,T]$ with 
\[
\begin{cases}
 k_1(t)=\mathbf{s}(\mathcal K_1)(t),  \\
 k_2(t)=\mathbf{s}(\mathcal K_2)(t),  \\
\mathbf{s}_1(t)=\mathbf{s}(h,\mathbf{s}(\mathcal K_1))(t),  \\ 
\mathbf{s}_2(t)=\mathbf{s}(h,\mathbf{s}(\mathcal K_2))(t)
\end{cases}
\] 
is a solution set of the system \eqref{system}.
\end{remark}
 
In the following two examples,  we also consider the Wiener space $C_0[0,1]$
and the Hilbert space $L_\infty[0,1]$ for simplicity.

\renewcommand{\theexample}{\thesection.5}

\begin{example}
Let 
$h_1=\sin \tfrac{\pi}{4}\big(t+\tfrac{1}{2} \big)$, 
$h_2=\cos \tfrac{\pi}{4}\big(t+\tfrac{1}{2} \big)$,
$h_3=\tan\tfrac{\pi}{4}\big(t+\tfrac{1}{2} \big)$,
$k_1(t)=\tan \tfrac{\pi}{4}\big(t+\tfrac{1}{2} \big)$, 
and
$k_2(t)= \sec \tfrac{\pi}{4}\big(t+\tfrac{1}{2} \big)\csc \tfrac{\pi}{4}\big(t+\tfrac{1}{2} \big)$ 
on $[0,1]$.
Then $\{h_1,h_2,h_3,k_1,k_2\}$ is a  set of functions in $L_\infty[0,1]$, and given the set $\mathcal H=\{h_1,h_2,h_3\}$, it 
follows that
\[
\begin{aligned}
\mathbf{s}(\mathcal H)(t)
&\equiv\mathbf{s}(h_1,h_2,h_3)^2(t)\\
&=\mathbf{s}\big(\sin\tfrac{\pi}{4}\big(\cdot+\tfrac{1}{2} \big),\cos\tfrac{\pi}{4}
\big(\cdot+\tfrac{1}{2} \big),\tan\tfrac{\pi}{4}\big(\cdot+\tfrac{1}{2} \big)\big)^2(t)\\
&=\sec^2 \tfrac{\pi}{4}\big(t+\tfrac{1}{2} \big)\\
&=k_1(t)k_2(t),
\end{aligned}
\] 
\[
\begin{aligned}
\mathbf{s}(\mathcal H,k_1)^2(t)
&\equiv \mathbf{s}(h_1,h_2,h_3,k_1)^2(t)\\
&=\sec^2\tfrac{\pi}{4}\big(t+\tfrac{1}{2} \big)+\tan^2\tfrac{\pi}{4}\big(t+\tfrac{1}{2} \big)
=\mathbf{s}(\mathbf{s}(\mathcal H),k_1)^2(t),
\end{aligned}
\] 
and
\[
\begin{aligned}
\mathbf{s}(\mathcal H,k_2)^2(t)
&\equiv  \mathbf{s}(h_1,h_2,h_3,k_2)^2(t)\\
&=\sec^2\tfrac{\pi}{4}\big(t+\tfrac{1}{2} \big) +\sec^2\tfrac{\pi}{4}\big(t+\tfrac{1}{2} \big)\csc^2\tfrac{\pi}{4}\big(t+\tfrac{1}{2} \big)\\
&=\mathbf{s}(\mathbf{s}(\mathcal H),k_2)^2(t),
\end{aligned}
\]
for all $t\in [0,1]$. 
 From this we see that
the set $\mathcal F_1=\{h, k_1, k_2, \mathbf{s}_1,\mathbf{s}_2\}$ 
of functions in $L_\infty[0,1]$ with 
\[
\begin{cases}
 h(t)= \mathbf{s}(h_1,h_2,h_3)(t)=\sec \tfrac{\pi}{4}\big(t+\tfrac{1}{2} \big),  \\
k_1(t)=\tan \tfrac{\pi}{4}\big(t+\tfrac{1}{2} \big) ,\\
k_2(t)= \sec \tfrac{\pi}{4}\big(t+\tfrac{1}{2} \big)\csc \tfrac{\pi}{4}\big(t+\tfrac{1}{2} \big),\\
\mathbf{s}_1(t)=\mathbf{s}(\mathcal H,k_1)(t),  \\ 
\mathbf{s}_2(t)=\mathbf{s}(\mathcal H,k_2)(t)
\end{cases}
\] 
is a solution set of the system \eqref{system}, and
equation \eqref{eq:multi-rel-01} holds with the sequence $\mathcal H=\{h_1,h_2,h_3\}$ and the  functions $k_1$ and $k_2$.
\end{example}
 
In the next example, the kernel functions of the Gaussian processes 
defining the transforms and convolutions involve trigonometric and hyperbolic (and hence exponential) functions.

\renewcommand{\theexample}{\thesection.6}

\begin{example}
Consider the function
\[
 h(t)= 2\sqrt{\csc\tfrac{\pi}{4}\big( t+\tfrac{1}{2}\big) \mathrm{cosh}\tfrac{\pi}{4}\big( t+\tfrac{1}{2}\big)}
\]
on $[0,1]$, and the  finite sequences
\[
\mathcal K_1=\big\{2\mathrm{tanh}\tfrac{\pi}{4}\big(t+\tfrac{1}{2} \big),
2\mathrm{sech} \tfrac{\pi}{4}\big(t+\tfrac{1}{2} \big),2 \cot\tfrac{\pi}{4}\big(t+\tfrac{1}{2} \big)\big\}
\]
and
\[
\mathcal K_2=\big\{\sqrt2\sin\tfrac{\pi}{4}\big(t+\tfrac{1}{2} \big),\sqrt2\cos\tfrac{\pi}{4}\big(t+\tfrac{1}{2} \big),
\sqrt2\mathrm{sinh}\tfrac{\pi}{4}\big(t+\tfrac{1}{2} \big),\sqrt2\mathrm{cosh}\tfrac{\pi}{4}\big(t+\tfrac{1}{2} \big)\big\}
\]
of functions in $L_{\infty}[0,1]$.
Then using the relationships  among hyperbolic functions and among  trigonometric functions, 
one can see that
\[
\mathbf{s}(\mathcal K_1)(t)=2\csc \tfrac{\pi}{4}\big(t+\tfrac{1}{2} \big)\quad \mbox{ and } 
\quad \mathbf{s}(\mathcal K_2)(t)=2\mathrm{cosh} \tfrac{\pi}{4}\big(t+\tfrac{1}{2} \big)
\]
on $[0,1]$.
From this we also see that
the set $\mathcal F_1=\{h, k_1, k_2, \mathbf{s}_1,\mathbf{s}_2\}$ 
of functions in $L_\infty[0,1]$ with 
\[
\begin{cases}
 h(t)= 2\sqrt{\csc\frac{\pi}{4}( t+\frac{1}{2}) \mathrm{cosh}\frac{\pi}{4}( t+\frac{1}{2})},  \\
k_1(t)=\mathbf{s}(\mathcal K_1)(t)=2\csc \tfrac{\pi}{4}\big(t+\tfrac{1}{2} \big) ,\\
k_2(t)=\mathbf{s}(\mathcal K_2)(t)=2\mathrm{cosh} \tfrac{\pi}{4}\big(t+\tfrac{1}{2} \big),\\
\mathbf{s}_1(t)=\mathbf{s}(h,\mathbf{s}(\mathcal K_1))(t),  \\ 
\mathbf{s}_2(t)=\mathbf{s}(h,\mathbf{s}(\mathcal K_2))(t)
\end{cases}
\] 
is a solution set  of the system \eqref{system}, and
equation \eqref{eq:multi-rel-02-2nd} holds with the function $h$, and the sequences $\mathcal K_1$   and $\mathcal K_2$.
\end{example}
 
\setcounter{equation}{0}
\section{Further results}

\par
In this section, we derive a more general relationship
between the iterated GFFT and the GCP for functionals 
in $\mathcal S(L_2[0,T])$.
To do this we also quote a  result  from \cite{Indag}.

\renewcommand{\thetheorem}{\thesection.1}

\begin{theorem} \label{thm:iter-gfft-more-1}
Let $F$ and $\mathcal H=\{h_1,\ldots,h_n\}$ be as in Theorem \ref{thm:2018-step1}. 
Assume that $q_1,q_2,\ldots, q_n$  are nonzero  real numbers with  
$\mathrm{sgn}(q_1)=\cdots=\mathrm{sgn}(q_n)$,  where `$\mathrm{sgn}$' denotes the 
sign function. Then, 
for all $p\in[1,2]$,
\[
\begin{aligned}
&T_{q_n,h_n}^{(p)}\big(T_{q_{n-1},h_{n-1}}^{(p)}\big(
   \cdots\big(T_{q_2,h_2}^{(p)}\big(T_{q_1,h_1}^{(p)}(F)\big)\big)\cdots\big)\big) (y)\\
&=T_{\alpha_n,\tau_n^{(n)}h_n}^{(p)}\Big(T_{\alpha_{n},\tau_n^{(n-1)}h_{n-1}}^{(p)}
\Big(\cdots \big(T_{\alpha_n,\tau_n^{(2)}h_2}^{(p)}
\big(T_{\alpha_n,\tau_n^{(1)}h_1}^{(p)}(F) \big)\big)\cdots\Big)\Big) (y)
\end{aligned}
\]
for s-a.e. $y\in C_{0}[0,T]$, where $\alpha_n$ is given by 
\[
\alpha_n=\frac{1}{\frac{1}{q_1}+\frac{1}{q_2}+\cdots+\frac{1}{q_n}}
\]
and $\tau_n^{(j)}=\sqrt{{\alpha_n}/{q_j}}$
for each $j\in \{1,\ldots,n\}$. Moreover it follows that
\[
T_{q_n,h_n}^{(p)}\big(T_{q_{n-1},h_{n-1}}^{(p)}\big(
   \cdots\big(T_{q_2,h_2}^{(p)}\big(T_{q_1,h_1}^{(p)}(F)\big)\big)\cdots\big)\big) (y)
=T_{\alpha_n,\mathbf{s}(\tau\mathcal H)}^{(p)}(F)(y)
\]
for s-a.e. $y\in C_{0}[0,T]$, where
$\mathbf{s}(\tau\mathcal H) \equiv \mathbf{s}(\tau_n^{(1)}h_1, \ldots, \tau_n^{(n)}h_n )$ 
is a function in $L_{\infty}[0,T]$ satisfying the relation
\[
\mathbf{s}(\tau\mathcal H) ^2(t)
=(\tau_n^{(1)}h_1)^2(t)+ \ldots+ (\tau_n^{(n)}h_n)^2(t)
\]
 for $m_L$-a.e. $t\in [0,T]$.
\end{theorem}

\par
Next, by a careful examination we see that for all $F\in \mathcal S(L_2[0,T])$ 
and any positive real $\beta>0$, 
\begin{equation}\label{eq:2018-new-parameter-change}
T_{\beta q,h}^{(p)}  (F) \approx T_{q,h/\sqrt{\beta}}^{(p)}(F) .
\end{equation}
Using  \eqref{eq:2018-new-parameter-change} and \eqref{eq:cp-fft-basic},
we have the following lemma.

\renewcommand{\thelemma}{\thesection.2}
\begin{lemma}  \label{thm:2018-last-pre}
Let $k_1$, $k_2$, $F$, $G$, and  $h$ be as in 
Theorem  \ref{thm:gfft-gcp-compose}. Let 
$q$, $q_1$, and $q_2$ be nonzero real numbers with 
$\mathrm{sgn}(q)=\mathrm{sgn}(q_1)=\mathrm{sgn}(q_2)$.
Then, for all $p\in [1,2]$,
\[
\begin{aligned}
&
\big(T_{q_1,\sqrt{q_1/(2q)} \mathbf{s}(h,k_1)}^{(p)} (F)*
 T_{q_2,\sqrt{q_2/(2q)}\mathbf{s}(h,k_2)}^{(p)}(G) \big)_{-q}^{(k_1,k_2)}(y)\\
&
=T_{q,h}^{(p)} \bigg(F \bigg(\frac{\cdot}{\sqrt2} \bigg)
G \bigg(\frac{\cdot}{\sqrt2}\bigg)\bigg)(y)
\end{aligned}
\]
for s-a.e. $y\in C_{0}[0,T]$.
\end{lemma}

\par
Finally, in view of Theorem  \ref{thm:iter-gfft-more-1} 
and Lemma  \ref{thm:2018-last-pre}, we obtain the following
assertion.

\renewcommand{\thetheorem}{\thesection.3}

\begin{theorem}  \label{thm:2018-last}
Let $k_1$, $k_2$, $F$, $G$, and  $h$ be as in 
Theorem \ref{thm:gfft-gcp-compose}. Let
$\mathcal H_1=\{h_{1j}\}_{j=1}^n$ 
and
$\mathcal H_2=\{h_{2l}\}_{l=1}^m$
be finite sequences of nonzero 
functions in $L_{\infty}[0,T]$. 
Given nonzero real numbers $q$, $q_1$, $q_{11}$, $\ldots$, $q_{1n}$, $q_2$,  $q_{21}$,
$\ldots$, $q_{2m}$ with 
\[
\begin{aligned}
\mathrm{sgn}(q)
&=\mathrm{sgn}(q_1)=\mathrm{sgn}(q_{11})=\cdots=\mathrm{sgn}(q_{1n})\\
&=\mathrm{sgn}(q_2)=\mathrm{sgn}(q_{21})=\cdots=\mathrm{sgn}(q_{2m}),
\end{aligned}
\]
let
 \[
\alpha_{1n}=\frac{1}{\frac{1}{q_{11}}+\frac{1}{q_{12}}+\cdots+\frac{1}{q_{1n}}},
\]
\[
\alpha_{2m}=\frac{1}{\frac{1}{q_{21}}+\frac{1}{q_{22}}+\cdots+\frac{1}{q_{2m}}},
\]
\[
\beta_{1n}=\frac{1}{\frac{1}{q_{1}}+\frac{1}{q_{11}}+\frac{1}{q_{12}}+\cdots+\frac{1}{q_{1n}}},
\]
and
\[
\beta_{2m}=\frac{1}{\frac{1}{q_{2}}+\frac{1}{q_{21}}+\frac{1}{q_{22}}+\cdots+\frac{1}{q_{2m}}}.
\]
Furthermore, assume that 
\[
h^2 = \mathbf{s}(\tau_{1n}\mathcal H_1)\mathbf{s}(\tau_{2m}\mathcal H_2)
\]
for  $m_L$-a.e. on $[0,T]$, where $\mathbf{s}(\tau_{1n}\mathcal H_1)$ 
and $\mathbf{s}(\tau_{2m}\mathcal H_2)$ are  functions in $L_{\infty}[0,T]$ 
satisfying the relation
\[
\mathbf{s}(\tau_{1n}\mathcal H_1)^2
\equiv \mathbf{s}(\tau_{1n}^{(1)}h_{11}, \ldots, \tau_{1n}^{(n)}h_{1n} )^2
=(\tau_{1n}^{(1)}h_{11})^2 + \cdots+ (\tau_{1n}^{(n)}h_{1n})^2 
\] 
and
\[
\mathbf{s}(\tau_{2m}\mathcal H_2)^2
\equiv \mathbf{s}(\tau_{2m}^{(1)}h_{21}, \ldots, \tau_{2m}^{(m)}h_{2m} )^2
=(\tau_{2m}^{(1)}h_{21})^2 + \cdots+ (\tau_{2m}^{(m)}h_{2m})^2 ,
\]
 respectively, and where
$\tau_{1n}^{(j)}=\sqrt{{\alpha_{1n}}/{q_{1j}}}$ for each $j\in \{1,\ldots,n\}$, and
$\tau_{2m}^{(l)}=\sqrt{{\alpha_{2m}}/{q_{2l}}}$ for each $l\in \{1,\ldots,m\}$.
For notational convenience, let
\[
h_1'=\sqrt{q_{1}/(2q)}h,\quad  h_{1j}'=\sqrt{\alpha_{1n}/(2q)}h_{1j}, \quad  j=1,\ldots,n, 
\]
and let
\[
h_2' =\sqrt{q_{2}/(2q)}h,\quad  h_{2l}'=\sqrt{\alpha_{2m}/(2q)}h_{2l}, \quad    l=1,\ldots,m.
\]
Then,  for all $p\in[1,2]$,
\[ 
\begin{aligned}
&\Big(T_{q_1,h_1'}^{(p)}\big(T_{q_{1n},h_{1n}'}^{(p)}\big(
\cdots\big(T_{q_{11},h_{11}'}^{(p)} (F)\big)\cdots \big)\big) \\
& \quad\quad\quad 
*T_{q_2,h_2'}^{(p)}\big( T_{q_{2m}, h_{2m}'}^{(p)}\big( 
\cdots\big(T_{q_{21}, h_{21}' }^{(p)} (G)\big)\cdots \big)\big)
\Big)_{-q}^{(\mathbf{s}(\tau_{1n}\mathcal H_1),\mathbf{s}(\tau_{2m}\mathcal H_2))} (y)\\
&=\Big(T_{q_1,\sqrt{q_{1}/(2q)} h}^{(p)}
\big( T_{\alpha_{1n},\sqrt{\alpha_{1n}/(2q)}\mathbf{s}(\tau_{1n}\mathcal H_1)}^{(p)}(F)\big)\\
&\quad \quad \quad
       *T_{q_2,\sqrt{q_{2}/(2q)}h}^{(p)}
\big(T_{\alpha_{2m},\sqrt{\alpha_{2m}/(2q)} \mathbf{s}(\tau_{2m}\mathcal H_2)}^{(p)}(G)\big)
\Big)_{-q}^{(\mathbf{s}(\tau_{1n}\mathcal H_1),\mathbf{s}(\tau_{2m}\mathcal H_2))} (y)\\
&=\Big( T_{\beta_{1n},\sqrt{\beta_{1n}/(2q)}\mathbf{s}(h,\mathbf{s}(\tau_{1n}\mathcal H_1))}^{(p)}(F)\\
&\qquad \qquad \qquad\qquad\,\,\,       
       *T_{\beta_{2m},\sqrt{\beta_{2m}/(2q)}\mathbf{s}(h,\mathbf{s}(\tau_{2m}\mathcal H_2))}^{(p)}(G)
\Big)_{-q}^{(\mathbf{s}(\tau_{1n}\mathcal H_1),\mathbf{s}(\tau_{2m}\mathcal H_2))} (y)\\
&= T_{q,h}^{(p)} \bigg(F \bigg(\frac{\cdot}{\sqrt2} \bigg)
 G \bigg(\frac{\cdot}{\sqrt2}\bigg)\bigg)(y) 
\end{aligned}
\] 
for s-a.e. $y\in C_{0}[0,T]$.
\end{theorem}
   
\section*{Acknowledgements}
 The authors would like to express their gratitude to the
editor and the referees for their valuable comments and suggestions which have
improved the original paper. The present research was supported supported by the research fund of Dankook University in 2019.


\end{document}